\documentclass[english,12pt,leqno]{amsart}

\usepackage{amsmath, amstext, amsfonts, amssymb, enumerate, enumitem, hyperref}

\usepackage{fullpage,times}

\newtheorem{thm}{Theorem}[section]

\newtheorem{lem}[thm]{Lemma}
\newtheorem{prop}[thm]{Proposition}
\newtheorem{ques}[thm]{Question}

\theoremstyle{definition}
\newtheorem{defn}[thm]{Definition}
\newtheorem{exe}[thm]{Example}
\newtheorem{rem}[thm]{Remark}
\newtheorem{back}[thm]{Background}

\newcommand{\Z}{\mathbf{Z}}

\newcommand{\R}{\mathbf{R}}

\newcommand{\Q}{\mathbf{Q}}
\newcommand{\G}{\mathbf{G}}

\newcommand{\SL}{\textnormal{SL}}
\newcommand{\Sp}{\textnormal{Sp}}
\newcommand{\cd}{\textnormal{cd}}

\newcommand{\BS}{\mathrm{BS}}

\newcommand{\ltb}{\beta_1}
\DeclareMathOperator{\DEF}{def}
\DeclareMathOperator{\Sol}{Sol}

\title[Presentability by products for some classes of groups]
{Presentability by products for some classes of groups}

\author{P.\ de la Harpe}
\address{Section de math\'ematiques,
Universit\'e de Gen\`eve,
C.P.~64,
1211 Gen\`eve~4, Switzerland}
\email{Pierre.delaHarpe@unige.ch}

\author{D.\ Kotschick}
\address{Mathematisches Institut, 
{\smaller LMU} M\"unchen,
Theresienstr.~39,
80333~M\"unchen, 
Germany}
\email{dieter@member.ams.org}

\thanks{Research by the second author done in part 
at the Institute for Advanced Study in Princeton 
with the support of The Fund For Math and The Oswald Veblen Fund.
We are grateful to M.~Bridson, M.~Bucher, Y.~de~Cornulier, G.~Levitt and C.~Weber
for useful information, and to L.~Paris for his permission to reproduce an unpublished result 
from his thesis in the appendix.}


\date{25 Sept.~2014, revised 4 March 2015 and 11 April 2016; 
\copyright{ P.\ de la Harpe and D.\ Kotschick 2014-2016}}

\begin{document}

\begin{abstract}
In various classes of infinite groups, we identify groups that are presentable by products, 
i.e.\ groups having finite index subgroups 
which are quotients of products of two commuting infinite subgroups.
The classes we discuss here include groups of small virtual cohomological dimension and
irreducible Zariski dense subgroups of appropriate algebraic groups. 
This leads to  applications to groups of positive deficiency, 
to fundamental groups of three-manifolds and to Coxeter groups.

For finitely generated groups presentable by products 
we discuss the problem of whether the factors in a 
presentation by products may be chosen to be finitely generated.
\end{abstract}

\maketitle

\section{Motivation}
\label{sectionMotivation}

For two connected closed oriented manifolds $M$, $N$ of the same dimension, 
a natural question is to ask whether $M$ \textbf{dominates} $N$,
i.e.\ whether there exists a continuous map $M \longrightarrow N$ of non-zero degree.
The interest in the ensuing transitive relation between homotopy types of manifolds
goes back at least to the late 1970's \cite{MiT--77, Gro--82, CaT--89}.

There are several known necessary conditions for $M$ to dominate $N$, 
some elementary, and some not.
On the elementary side, if there exists an $f \colon M \longrightarrow N$ of non-zero degree, 
then $f^*$ injects the rational cohomology ring of $N$ into that of $M$, 
in particular the Betti numbers of $M$ are at least as large as the Betti numbers of $N$, 
and $f_*$ surjects $\pi_1(M)$ onto a finite index subgroup of $\pi_1(N)$; see~\cite{Hop--30}, 
Satz~I\footnote{Hopf~\cite{Hop--30} did not use cohomology, but formulated the result using the 
Umkehr map on intersection rings.} resp.~\S~6. 
Non-elementary obstructions to domination often arise from geometric considerations, 
for example through the use of harmonic maps~\cite{CaT--89}, 
or the application of asymptotic invariants of manifolds, 
such as the simplicial volume \cite{Gro--82} or the minimal volume entropy \cite{BCG--95}.

The subject of the present paper is motivated by a particular case of the domination question. 
See \cite[Page 304]{Gro--99},
where Gromov discusses the issue of representing even-degree homology classes
by products of surfaces, and \cite{KoL--09}.

\begin{ques}
\label{particularques}
For a given connected closed oriented $n$-manifold $N$, 
does there exist a non-trivial direct product 
$M = M_1 \times M_2$ of dimension $n$ which dominates $N$~?
\end{ques}

The answer is easy in dimension two: 
$\Sigma_g$ dominates $\Sigma_h$ if and only if $g \ge h$,
where $\Sigma_g$ denotes the closed orientable connected surface of genus $g$;
in particular $\Sigma_h$ is dominated by a non-trivial product 
(i.e. by the two-torus) if and only if $h=0$ or $h=1$.
An answer is also known in dimension three; see \cite{KoN--13}.
Simple general answers cannot be expected in higher dimensions, 
but partial results are contained in \cite{KoL--09,KoL--13,Neo--14}.

Before formulating another answer, it is convenient 
to introduce two definitions.
The first one originates in \cite{Gro--82, Gro--83}.

\begin{defn}
\label{defratess}
Let $N$ be a closed oriented connected manifold of dimension $n$.
Denote by $[N] \in H_n(N, \Q)$ its fundamental class,
by $\Gamma$ its fundamental group,
by $c \colon N \longrightarrow \operatorname{B}\Gamma$ 
the classifying map of the universal covering,
and by $c_n$ the induced map on the rational homology groups $H_n(-, \Q)$.
The manifold $N$ is said to be \textbf{rationally essential}
if $c_n([N]) \ne 0$ in $H_n( \operatorname{B}\Gamma, \Q)$.
\end{defn}

The obvious examples of rationally essential manifolds are the aspherical ones,
i.e.\ manifolds  with contractible universal covering.
Other examples include manifolds with non-zero simplicial volume
\cite[Corollary B, Section 3.1]{Gro--82}, 
manifolds satisfying suitable enlargeability conditions in the sense of Gromov--Lawson
\cite{HKRS--08},
and manifolds that have non-zero degree maps onto rationally essential manifolds,
in particular connected sums with a rationally essential summand \cite[Page 3]{Gro--83}.

The next definition is from \cite{KoL--09}:

\begin{defn}
\label{defppgps}
Let $\Gamma$ be an infinite group.
A \textbf{presentation  of $\Gamma$ by a product} is a homomorphism
$\varphi \colon \Gamma_1 \times \Gamma_2 \longrightarrow \Gamma$,
where $\Gamma_1, \Gamma_2$ are two groups,
such that $\varphi(\Gamma_1 \times \{1\}), \varphi(\{1\} \times \Gamma_2)$ are infinite
and $\varphi(\Gamma_1 \times \Gamma_2)$ is of finite index in $\Gamma$.
The group $\Gamma$ is \textbf{presentable by a product}
if there exists such a homomorphism.
\end{defn}

Clearly, an infinite group $\Gamma$ is presentable by a product
if and only if there exist two commuting infinite subgroups of $\Gamma$
whose union generates a subgroup of finite index.
Note that the images of $\Gamma_1$ and $\Gamma_2$
need not have trivial intersection, indeed need not be distinct;
for example, $\varphi \colon \Z \times \Z \longrightarrow \Z$, 
$(m,n) \longmapsto m+n$, shows that an infinite cyclic group is presentable by a product.

\begin{exe}
\label{basicexplspp}
We collect here some basic examples for this notion.
Recall that a group $\Gamma$ is \textbf{virtually} another group $\Delta$
if it has a finite index subgroup isomorphic to $\Delta$.
\begin{enumerate}[noitemsep,label=(\arabic*)]
\item\label{1DEbasicexplspp}
A group with an infinite centre is presentable by a product.
\item\label{2DEbasicexplspp}
Let $\Gamma$ be an infinite group in which every subgroup of finite index has trivial centre.
Then $\Gamma$ is presentable by a product if and only if $\Gamma$ is virtually a direct 
product of two infinite groups \cite[Proposition 3.2]{KoL--09}.
\item\label{3DEbasicexplspp}
Presentability by products is invariant under passage to finite index subgroups
\cite[Lemma 3.4]{KoL--09}. 
\item\label{4DEbasicexplspp}
A free product $\Gamma_1 \ast \Gamma_2$ of two non-trivial groups
is presentable by a product if and only if both factors are of order two;
see \cite[Corollary 9.2]{KoL--13}.
\item\label{5DEbasicexplspp}
An infinite Gromov hyperbolic group is not presentable by a product,
unless it is elementary, i.e.\ unless it has an infinite cyclic subgroup of finite index;
see \cite[Theorem 1.5]{KoL--09}.
\item\label{6DEbasicexplspp}
Infinite simple groups are not presentable by products.
\end{enumerate}
The references given for 
\ref{2DEbasicexplspp}, \ref{3DEbasicexplspp} and \ref{4DEbasicexplspp}
deal with countable groups only, but the arguments carry over to the general case.
\end{exe}

The following result, Theorem 1.4 in \cite{KoL--09}, is an attempt to 
answer Question \ref{particularques} for rationally essential manifolds. 

\begin{thm}
\label{translation}
Let $N$ be a rationally essential closed oriented manifold. If $N$ is dominated by a 
non-trivial product $M_1\times M_2$, then its fundamental group $\pi_1(N)$ is presentable by a product.
\end{thm}

The proof of this theorem is essentially a diagram chase using the universal properties of 
classifying maps, and it gives a presentation of $\pi_1(N)$ by a product in which 
the groups $\Gamma_1$ and $\Gamma_2$ are (the images of) $\pi_1(M_1)$ and 
$\pi_1(M_2)$. In particular, since the $M_i$ are compact manifolds, the $\Gamma_i$ are finitely generated.
Therefore, the conclusion 
of Theorem~\ref{translation} can be strengthend: not only is $\pi_1(N)$ presentable by a product,
it is presentable by a product of finitely generated groups.

\subsection*{Overview}
In this paper we discuss presentability by products for several classes of groups, 
and we often consider arbitrary groups which need not be finitely generated, 
and not even countable. 
However, because of the (strengthening of the) theorem above, 
for finitely generated groups presentable by products it is an interesting question
whether the factors in a presentation by a product can be taken to be finitely generated. 
We will show in Section~\ref{s:Abels} that there are indeed examples of finitely generated groups 
that are presentable by products, but the factors cannot both be chosen to be finitely generated
(Theorem \ref{AlaAbels}).

In Section~\ref{sectionSchreier} we introduce the Schreier property for  groups, 
which is  motivated by Schreier's classical theorem~\cite{Schreier} 
on finitely generated subgroups of free groups.
The main new result is that one-ended groups with the Schreier property 
are not presentable by products of finitely generated groups (Theorem \ref{thSchreier}).

In Section~\ref{sectionCohom2} we discuss 
necessary and sufficient conditions 
for groups in some classes to be presentable by products.
More precisely, let $\Gamma$ be a finitely presented infinite group.
If $\Gamma$ is of virtual cohomological dimension at most $2$,
it is presentable by a product if and only if it is virtually either infinite cyclic
or isomorphic to a product of free groups $F_k \times F_\ell$ with $k,\ell \ge 1$.
If $\Gamma$ is of positive deficiency,
$\Gamma$ is presentable by a product if and only if it is virtually either infinite cyclic
or $F_k \times \Z$ with $k \ge 1$
(Theorems \ref{t:cd2} and \ref{t:def}).

In Section~\ref{sectionZariski} we show that Zariski dense subgroups 
of suitable algebraic groups (over fields of characteristic $0$)
are not presentable by products.

Section~\ref{s:applic} contains applications
to Baumslag-Solitar groups,
to fundamental groups of three-manifolds, 
and to Coxeter groups.

\section{Finite generation of the factors in presentations by products}
\label{s:Abels}

Let $\Gamma_1$ and $\Gamma_2$ be subgroups of $\Gamma$ that commute and 
generate $\Gamma$. 
Then the intersection $\Gamma_1\cap\Gamma_2$ is in the centre of $\Gamma$. 
Moreover, the $\Gamma_i$ are normal in $\Gamma$ and we have exact sequences 
\begin{alignat*}{2}
    1\longrightarrow \Gamma_1
     &\longrightarrow \Gamma
     &&\longrightarrow \Gamma_2/(\Gamma_1\cap\Gamma_2) 
     \longrightarrow 1 \text{\makebox[0pt][l]{\ ,}}
\\
    1\longrightarrow \Gamma_1\cap\Gamma_2
     &\longrightarrow \Gamma_2
     &&\longrightarrow \Gamma_2/(\Gamma_1\cap\Gamma_2) 
     \longrightarrow 1 \text{\makebox[0pt][l]{\ ,}}
  \end{alignat*}
and similarly with the roles of $\Gamma_1$ and $\Gamma_2$ reversed.

If we assume $\Gamma$ to be finitely generated, 
the first sequence shows that the quotient $\Gamma_2/(\Gamma_1\cap\Gamma_2)$
is finitely generated. 
Now if, in addition, the intersection $\Gamma_1\cap\Gamma_2$ is finitely generated, 
then the second sequence shows that $\Gamma_2$ is finitely generated. 
We conclude that if $\Gamma$ is finitely generated and has finitely generated centre,
then, in every presentation by a product, 
the two factors $\Gamma_1$ and $\Gamma_2$ are also finitely generated.

If we have a finitely generated $\Gamma$ whose centre is not finitely generated, but contains an element of infinite order, then this 
element generates a central infinite cyclic subgroup $\Gamma_1\subset\Gamma$, 
and we can take $\Gamma_2=\Gamma$ to obtain a presentation of $\Gamma$ 
by a product of finitely generated groups. 
Therefore, the only finitely generated groups that are candidates 
for being presentable by products, 
but never with finitely generated factors $\Gamma_i$, 
are groups whose centre is an infinitely generated torsion group.

Using a construction of Hall~\cite{Hal--61} and Abels~\cite{Abe--77}, we can indeed produce such examples.

\begin{thm}
\label{AlaAbels}
There exists a finitely generated group $\Gamma$ that is presentable by products, 
but in every such presentation at least one of the factors
$\Gamma_i$ is not finitely generated.
\end{thm}

\begin{proof}
Let $p$ be a prime number. We denote by $A_3$ the group of matrices of the  form
$$
\begin{pmatrix}
1 & x & z\\
0 & u & y\\
0 & 0 & 1
\end{pmatrix}
 \ \ \ \ \textrm{with} \ \ x, y, z\in\Z[1/p] , \ u\in\Z[1/p]^\times \ .
$$
This group is finitely generated~\cite{Abe--77,AbB--87}. 
Its centre consists of the matrices satisfying $u=1$ and $x=y=0$. 
Therefore, the centre is isomorphic to the additive group $(\Z[1/p],+)$ 
via the map that sends a matrix to its
upper-right-hand entry $z$. 
In particular, the centre is not finitely generated.

The infinite cyclic group $(\Z,+)$ is embedded in the centre of $A_3$ 
as the elements for which $z\in\Z$.
Define $\Gamma = A_3/\Z$ for this central embedding $\Z\subset A_3$. 
Then $\Gamma$ is finitely generated because $A_3$ is, 
and the centre of $\Gamma$ is isomorphic to the additive group $\Z[1/p]/\Z$, i.e.~an 
abelian torsion group that is not finitely generated.

Since $\Gamma$ has infinite centre $C(\Gamma)$, 
it is presentable by the product $C(\Gamma)\times\Gamma$.

Next, let $\Gamma_1$ and $\Gamma_2$ be 
a pair of commuting infinite subgroups in $\Gamma$ for which the multiplication
$\varphi\colon\Gamma_1\times\Gamma_2\longrightarrow\Gamma$ 
has image of finite index in $\Gamma$. 
If the intersection $\Gamma_i\cap C(\Gamma)$ has finite index in $\Gamma_i$, 
then it is infinite. As an infinite subgroup
of the abelian torsion group $C(\Gamma)$ it cannot be finitely generated, 
and so $\Gamma_i$ is also not finitely generated. 
Thus, the only possibility for both $\Gamma_i$ to be finitely generated is for
$\Gamma_i\cap C(\Gamma)$ to be of infinite index in $\Gamma_i$ for both $i=1$ and $i=2$. 
In this case the images of the $\Gamma_i$ in $\Gamma/C(\Gamma)$ are both infinite, 
and of course they commute and generate a subgroup of finite index, 
leading to the conclusion that $\Gamma/C(\Gamma)$ is presentable by a product.
The proof of the theorem is therefore completed by the next lemma.
\end{proof}

\begin{lem}
\label{lemmaforAbels}
The group $\Gamma/C(\Gamma)$ is not presentable by products,
where $\Gamma = A_3 / \Z$ as in the previous proof.
\end{lem}

\begin{proof}
This is an application of Proposition~3.2~in \cite{KoL--13},
where it was proved that groups containing infinite acentral subgroups of infinite index
are not presentable by products.
A subgroup $G$ of a group $H$ is \textbf{acentral} if, for every $g \in G \smallsetminus \{1\}$,
the centralizer $C_H(g)$ is contained in $G$.

In the present case, we choose for the subgroup $G$ of $\Gamma/C(\Gamma)$
the image of the subgroup of diagonal matrices in $A_3$.
Every $g\neq e\in G$ is represented by a matrix of the form 
$$
\begin{pmatrix}
1 & 0 & 0\\
0 & \pm p^n & 0\\
0 & 0 & 1
\end{pmatrix} \ \ \textrm{with} \ \ n\in\Z\setminus\{ 0\} \ .
$$
If a matrix in $A_3$ represents an element of the centralizer of $g$ in $\Gamma/C(\Gamma)$, 
then a straightforward check shows that in that matrix $x=y=0$. 
Therefore, the image of that matrix in $\Gamma/C(\Gamma)$ lies in $G$.
Hence $G$ is an infinite acentral subgroup of infinite index in $\Gamma/C(\Gamma)$.
\end{proof}

\section{Schreier groups}
\label{sectionSchreier}

In 1927, Schreier \cite{Schreier} proved that 
every finitely generated non-trivial normal subgroup 
of a free group has finite index. 
The result was extended to surface groups of 
genus at least two by Griffiths \cite{Griffiths},
to non-trivial free products by Baumslag \cite{Bau--66},
and to some free products with amalgamation by Karrass and Solitar 
(see \cite{KaS--73} and the references given there).
These results motivate the following definition.

\begin{defn}
\label{defSchreier}
A group is a \textbf{Schreier group}, 
or has the \textbf{Schreier property}, 
if every finitely generated normal  subgroup is either finite or of finite index.
\end{defn}

The following definition appears in \cite{KaS--73}:
a group has the \emph{finitely generated normal property}
if all its non-trivial finitely generated normal subgroups are of finite index.
The groups mentioned above, from papers by Schreier and others,
do not have non-trivial finite normal subgroups at all,
and have the finitely generated normal property.
But we insist that a Schreier group may have non-trivial finite normal subgroups.
The definition of~\cite{KaS--73} was rediscovered in~\cite{Cat--03}, 
where the results of Schreier and Griffiths are reproved, without any 
reference to the earlier literature.

Before listing some examples of Schreier groups, we record the following.

\begin{lem}
\label{prop:S} 
Let $\Gamma$ be a group.
\begin{enumerate}[noitemsep,label=(\arabic*)]
\item\label{1DEprop:S}
Let $\Gamma_0$ be a subgroup of finite index in $\Gamma$.
Then $\Gamma_0$ is a Schreier group if and only if $\Gamma$ is.
\item\label{2DEprop:S}
Let $F$ be a finite normal subgroup in $\Gamma$.
Then $\Gamma/F$ is a Schreier group if and only if $\Gamma$ is.
\end{enumerate}
\end{lem}

\begin{proof}
All implications are straightforward.
Let us write the details of part of \ref{2DEprop:S},
and leave it to the reader to check the other implications.
We assume that $\Gamma/F$ has the Schreier property,
we consider a normal  finitely generated infinite subgroup $N$ of $\Gamma$,
and we have to show that the index of $N$ in $\Gamma$ is finite.

Inside $\Gamma / F$, the quotient $N /(F \cap N)$ is a subgroup
which is normal, finitely generated, and infinite.
Since $\Gamma / F$ has the Schreier property, $N / (F \cap N)$ is of finite index in $\Gamma / F$.
It follows that $N$ is of finite index in $\Gamma$.
\end{proof}

\begin{exe}
\label{exSchreier}
Here are some examples of groups with the Schreier property.
\begin{enumerate}[noitemsep,label=(\arabic*)]
\item\label{1DEexSchreier}
Two-ended groups are virtually infinite cyclic, and therefore have the Schreier property.
\item\label{2DEexSchreier}
Every irreducible lattice $\Gamma$ in a connected semi-simple Lie group 
with finite centre and real rank $\geq 2$ 
has the Schreier property by the Margulis normal subgroup theorem. 
(There is a much stronger result in \cite[Chap.~4, Section~4]{Mar--91}.)
Note that the Margulis theorem implies more: every normal subgroup $\Gamma$
either is finite or has finite index, and is therefore finitely generated \emph{a posteriori.} 
\item\label{3DEexSchreier}
Every finitely generated
group with positive first $\ell^2$-Betti number is a Schreier group 
by a result of Gaboriau \cite{gaboriauIHES},
generalizing the work of L\"uck \cite{lueckMA}. 
This class of groups encompasses the original cases 
considered by Schreier and Griffiths, 
and many others, such as groups with infinitely many ends,
groups with deficiency $\geq 2$, and non-abelian limit groups in the sense of Sela.
\end{enumerate}
\end{exe}

While the Schreier property arises naturally in other contexts, 
for example in the study of K\"ahler groups, compare~\cite{Cat--03,Kot--12},
our interest in it here stems from the following result.

\begin{thm}
\label{thSchreier}
A finitely generated group with the Schreier property 
is presentable by a product of finitely generated groups
if and only if it is  virtually infinite cyclic, equivalently, if it is two-ended.
\end{thm}

\begin{proof}
Clearly every two-ended group is presentable by a product and has the Schreier property, 
cf.\ Example \ref{exSchreier}\ref{1DEexSchreier}. 
Conversely, let $\Gamma$ be a finitely generated group with the Schreier property
that is presentable by a product of finitely generated infinite subgroups $\Gamma_1$ and $\Gamma_2$.
After replacing $\Gamma$ by a finite index subgroup, denoted again by $\Gamma$, we may assume that the multiplication $\Gamma_1\times\Gamma_2\longrightarrow\Gamma$
is surjective. It follows that the $\Gamma_i$ are normal in $\Gamma$. Since they are finitely generated,
 the Schreier property implies that they are of finite index in $\Gamma$. 
Therefore the intersection $\Gamma_1\cap\Gamma_2$ is also of finite index in $\Gamma$, 
and since $\Gamma_1\cap\Gamma_2$ is central 
we conclude that $\Gamma$ is virtually abelian. 
We pass to a finite index abelian subgroup. 
Since this is infinite, finitely generated, and  has the Schreier property, 
its rank is one, and so $\Gamma$ is virtually infinite cyclic.
\end{proof}

In view of Examples~\ref{exSchreier}\ref{2DEexSchreier} 
and~\ref{exSchreier}\ref{3DEexSchreier}, 
Theorem~\ref{thSchreier} is a common generalization
of Propositions~4.1 and 8.1 in~\cite{KoL--13}. 
One point of this generalization is that it is likely that other criteria 
will emerge in the future that guarantee the Schreier property,
other than the positivity of the first $\ell^2$-Betti number. 
In special cases it is already known that positivity of the rank gradient 
(in the sense of Lackenby), 
or cost strictly larger than $1$ (in the sense of Levitt and Gaboriau) 
imply the Schreier property; compare with Proposition 13 in \cite{AbN--12}.

Several proofs were given in~\cite{KoL--13} 
for the fact that groups with infinitely many ends are not presentable by products. 
The new content of Theorem~\ref{thSchreier} is in showing that one-ended groups
with the Schreier property are not presentable by products of finitely generated groups.

\section{Groups of cohomological dimension at most two}
\label{sectionCohom2}

The \textbf{cohomological dimension} $\cd (\Gamma)$ of a group $\Gamma$
is the maximum (possibly $\infty$) of the integers $n$ such that $H^n(\Gamma, A) \ne 0$
for some $\Z[\Gamma]$-module $A$.

\begin{exe}
\label{excohom2}
Here are two classes of basic examples:
\begin{enumerate}[noitemsep,label=(\arabic*)]
\item\label{1DEexcohom2}
Surface groups are of cohomological dimension $2$. More generally,
by Lyndon's theorem, torsion-free one-relator groups are of cohomological at most $2$;
cf.~\cite[Theorem 7.7]{Bie--81}.
\item\label{2DEexcohom2}
Torsion-free fundamental groups of compact three-manifolds with non-empty non-spherical 
boundaries are of cohomological dimension at most $2$.
In particular, groups of non-trivial knots are of cohomological dimension $2$.
\end{enumerate}
\end{exe}

The next theorem is a consequence of results of Bieri \cite{Bie--76, Bie--81}.

\begin{thm}
\label{t:cd2}
An infinite finitely presented group $\Gamma$ of virtual cohomological dimension at most two
is presentable by a product if and only if it is either virtually infinite cyclic 
or virtually $F_k \times F_l$ with $k$, $l\geq 1$. 
\end{thm}

As usual, $F_k$ denotes a free group on $k$ generators; note that $F_1 = \Z$.

\begin{proof}
If $\Gamma$ is of virtual cohomological dimension at most $2$ we replace it by a finite index
subgroup that is of actual cohomological dimension at most $2$. This does not affect 
presentability by products as recalled in Example \ref{basicexplspp}.

Our discussion of groups of cohomological dimension $2$ is based on the results of Bieri first 
published in~\cite{Bie--76}. 
We shall use the more leisurely and polished \cite{Bie--81} as our reference.

Suppose $\Gamma$ is a non-trivial finitely presented group 
of cohomological dimension at most $2$. Then $\Gamma$ is torsion-free.
The cohomological dimension is invariant under further passages to subgroups of finite index by Serre's theorem, 
cf.~\cite[Theorem~5.11]{Bie--81}.
Thus we may freely replace
$\Gamma$ by subgroups of finite index as needed.

Let $\Gamma_1$, $\Gamma_2 \subset \Gamma$ 
be commuting infinite subgroups  so that the multiplication 
\begin{equation*}
\varphi\colon\Gamma_1 \times \Gamma_2 \longrightarrow \Gamma
\end{equation*}
is surjective. 
(Here we already tacitly pass to a subgroup of finite index to get surjectivity.) 

If $\Gamma_1 \cap \Gamma_2$ is trivial, then $\varphi$ is an isomorphism 
by \cite[Lemma~3.3]{KoL--09}. 
It follows that both $\Gamma_i$ are finitely presented and of cohomological dimension at most $2$.
Since both $\Gamma_i$ are normal and of infinite index in $\Gamma$, 
they are free by a result of Bieri \cite[Corollary~8.6]{Bie--81}.
Thus $\varphi^{-1}$ is an isomorphism
 between $\Gamma$ and $F_k \times F_\ell$ for some $k, \ell \geq 1$.

If $\Gamma_1 \cap \Gamma_2$ is non-trivial, then it is infinite since $\Gamma$ is torsion-free. 
Moreover, $\Gamma_1 \cap \Gamma_2$ is in the centre of $\Gamma$
 by \cite[Lemma~3.3]{KoL--09}. 
Thus $\Gamma_1 \cap \Gamma_2$ is free abelian of rank one or two. 
If the rank is $=2$, then $\Gamma$ is abelian 
and itself isomorphic to $\Z^2$ by \cite[Corollary~8.9]{Bie--81}.
If $\Gamma_1 \cap \Gamma_2 = \Z$, 
then another result of Bieri \cite[Corollary~8.7]{Bie--81} shows that the quotient of 
$\Gamma$ by the central subgroup $\Gamma_1 \cap \Gamma_2$ 
is virtually a finitely generated free group. 
Passing to a finite index subgroup we may assume that the quotient is free, 
and then, passing to a finite index subgroup once more, 
we see that the central extension
\begin{equation*}
1 \longrightarrow \Z \longrightarrow \Gamma \longrightarrow F_k \longrightarrow 1 ,
\end{equation*}
splits.

We have now proved that if a finitely presented group $\Gamma$ 
of cohomological dimension $\leq 2$  is presentable by a product, 
then it is virtually $F_k \times F_\ell$ for some $k, \ell \geq 0$, but not both $=0$. 
If either $k$ or $\ell$ vanishes, then  $\Gamma$ is virtually free. 
Since it is presentable by a product, it is virtually $\Z$ 
(see Example \ref{basicexplspp}). 
Thus, our conclusion is that $\Gamma$ is virtually infinite cyclic, 
or virtually a product  $F_k \times F_\ell$ with both $k, \ell \geq 1$. 
Conversely, every group of this form is presentable by products.
This completes the proof of Theorem~\ref{t:cd2}.
\end{proof}

We now apply Theorem~\ref{t:cd2} to groups of positive deficiency.
The \textbf{deficiency} of a finitely presented group $\Gamma$
is the maximum $\DEF (\Gamma)$
over all presentations of the difference of the number of generators 
and the number of relators
(this maximum is finite, by \cite[Lemma 1.2]{Eps--61}).

\begin{exe}
\label{exdeficiency}
Here are some basic examples:
\begin{enumerate}[noitemsep,label=(\arabic*)]
\item\label{1DEexdeficiency}
For non-abelian free groups: $\DEF (F_k) = k$.
\item\label{2DEexdeficiency}
$\DEF (\Gamma) = m-1$ for a group $\Gamma$
having a presentation with $m$ generators and one non-trivial relation \cite[Lemma 1.7]{Eps--61}.
\item\label{3DEexdeficiency}
Groups of classical knots have deficiency $1$. More generally,
$\DEF (\pi_1(M))) \ge 1$ for every oriented connected compact three-manifold $M$ with non-empty 
non-spherical boundary \cite[Lemma 2.2]{Eps--61}.
\item\label{4DEexdeficiency}
By (2) above, an orientable surface group of genus $g$ has deficiency $2g-1$. This still holds
for orbifold groups of genus $g$ by \cite[Lemma~2]{Kot--12}, although these are not one-relator groups
in general.
\end{enumerate}
\end{exe}

The deficiency has the following relationship to the first $\ell^2$-Betti number $\ltb$.

\begin{prop}[\cite{Hil--97}]
\label{p:hill}
For every finitely presented group $\Gamma$, we have
\begin{equation}
\tag{*}
\label{(*)}
\DEF (\Gamma) \leq 1+ \ltb (\Gamma) \ .
\end{equation}
If equality holds, then 
the presentation complex of a presentation realizing the deficiency is aspherical.
\end{prop}

The inequality~\eqref{(*)} has been part of the folklore for a long time, 
and is a special case of the Morse inequalities in the $\ell^2$ setting. 
For the case of equality we refer to the paper by Hillman~\cite{Hil--97}.

\begin{thm}
\label{t:def}
An infinite finitely presented group $\Gamma$ of positive deficiency
is presentable by a product if and only if it is 
either infinite cyclic or virtually $F_k \times \Z$ with $k\geq 1$. 
If one of these conditions holds, then the deficiency of $\Gamma$ is $1$.
\end{thm}
\begin{proof}
If $\Gamma$ is presentable by a product, 
then $\ltb (\Gamma)=0$ by \cite[Proposition~4.2]{KoL--13}. 
It follows that the  deficiency is at most $1$, 
and since we assume 
positive deficiency and presentability by products, 
we only have to consider groups $\Gamma$ of deficiency $=1$ with equality in~\eqref{(*)}. 
The characterization of the  equality case by Hillman~\cite{Hil--97} 
tells us that the cohomological dimension of $\Gamma$ is at most $2$, and so Theorem~\ref{t:cd2} is applicable.
Moreover, $\Gamma$ must be torsion-free.

If $\Gamma$ is virtually infinite cyclic, then, since it is torsion-free, it is itself infinite cyclic.
If $\Gamma$ is not virtually infinite cyclic, 
then by Theorem~\ref{t:cd2} it is virtually $F_k\times F_\ell$ with both $k, \ell \geq 1$.
If both $k, \ell \geq 2$ then the Morse inequality and the K\"unneth formula give us
\begin{equation*}
\DEF (F_k \times F_\ell) \leq b_1(F_k\times F_\ell) - b_2(F_k \times F_\ell) 
\, = \, k + \ell - k  \ell \leq 0 
\end{equation*}
(where $b_r(-) := \dim (H_r(-, \Q)$ denotes a Betti number).
Thus $\Gamma$ has a finite index subgroup of non-positive deficiency. 
This contradicts $\DEF (\Gamma)=1$,
since if $\Gamma$ has a presentation with $a$ generators and $b$ relations, 
then, by the Reidemeister--Schreier process, 
a subgroup $\bar\Gamma\subset\Gamma$ of index $d$ 
has a presentation with $\bar a = (a-1)d+1$ generators 
and $\bar b = bd$ relations. 
Therefore
\begin{equation*}
\DEF (\bar\Gamma )\geq \bar a-\bar b = (a-b-1)d+1\geq 1
\end{equation*}
because we may take $a-b=1$. 
Thus we conclude that $\Gamma$ is virtually $F_k\times\Z$ with $k\geq 1$.

Conversely, both $\Z$ and every group $\Gamma$ that is virtually $F_k\times\Z$ 
with $k\geq 1$ are presentable by products. 
Moreover, it is easy to see that $\Z$ and $F_k\times\Z$ with $k\geq 1$ are both of deficiency $=1$.
By the above argument, every group that is virtually $F_k\times\Z$ has deficiency at most~$1$.
\end{proof}

\begin{rem}
As pointed out to us by the anonymous referee, the groups appearing in the conclusion of Theorem~\ref{t:def},
i.e.~the groups that are torsion-free and virtually $F_k\times\Z$ for some $k$, are well understood from 
several points of view; see for example~\cite{Kro--90,Lev--15} and the references given there. They are 
the so-called unimodular generalized Baumslag-Solitar groups. We will discuss the special case of 
Baumslag-Solitar groups (rather than generalized ones) in Subsection~\ref{subsectionBS} below.
\end{rem}

\section{Zariski dense subgroups of algebraic groups}
\label{sectionZariski}

When a group $\Gamma$ is naturally a subgroup of a larger group $G$,
e.g.\ a Lie group or an algebraic group of some sort, 
properties of $\Gamma$ can often be deduced from properties of $G$.
Theorems \ref{inGalgnotproduct} and \ref{irrZdensenot} below 
are illustrations of this for presentability by products.
Before stating them, we need some preliminaries.

Let $k$ be a field, given as a subfield of an algebraically closed field $K$.
Let $\mathfrak g_K$ be a Lie algebra over $K$.
 A $k$-form of $\mathfrak g_K$
is a Lie algebra $\mathfrak g_k$ over $k$ 
given together with an isomorphism 
$\mathfrak g_k \otimes_k K \overset{\simeq}{\longrightarrow} \mathfrak g_K$.
A $k$-Lie algebra is a Lie algebra over $K$
given together with a $k$-form.
Let $\G$ be a linear algebraic $k$-group.
We denote by $\G (k)$
the group of $k$-points of $\G$,
and by $\mathfrak g$ its Lie algebra, viewed as a $k$-Lie algebra.

\begin{defn}
\label{almostdirectproduct}
(1)
A connected linear algebraic $k$-group $\G$
is \textbf{presentable by a product} if there exist
two commuting connected closed $k$-subgroups 
$\G_1, \G_2$ of $\G$ of positive dimensions
such that the multiplication homomorphism
$\mu : \G_1 \times \G_2 \longrightarrow \G$, $(g_1, g_2) \longmapsto g_1g_2$
is onto.

Note that if $\G_1, \G_2$ are as above, 
then they are normal subgroups of $\G$, both non-trivial,
and each of them is contained in the centralizer of the other.

(2)
A $k$-Lie algebra $\mathfrak g$ is \textbf{presentable by a product} if there exist
two commuting subalgebras
$\mathfrak g_1, \mathfrak g_2$ of $\mathfrak g$ of positive dimensions
such that the homomorphism $\mathfrak g_1 \times \mathfrak g_2 \longrightarrow \mathfrak g$,
$(X_1 , X_2) \longmapsto X_1 + X_2$ is onto.

Note that if $\mathfrak g_1, \mathfrak g_2$ are as above, 
then they are non-zero ideals in $\mathfrak g$,
and each one is contained in the centralizer in $\mathfrak g$ of the other.
\end{defn}

From now on, we will systematically assume that the characteristic
of the ground field is zero. 
This has useful consequences for our purpose:

\begin{thm}
\label{motivchar0}
Let $k$ be a field of characterisitic zero,
$\G$ a connected linear algebraic $k$-group,
and $\mathfrak g$ its Lie algebra.
\begin{enumerate}[noitemsep,label=(\arabic*)]
\item\label{1DEmotivchar0}
$\G (k)$ is infinite if and only if the dimension of $\G$ is positive.
\item\label{2DEmotivchar0}
If $\mathfrak g$ is not presentable by a product,
then $\G$ is not presentable by a product.
\end{enumerate}
\end{thm}

\begin{proof}
For \ref{1DEmotivchar0}, 
which holds more generally over perfect fields,
see e.g.\ \cite[Corollary 13.3.10]{Spr--09}.
Note that there are counter-examples in (non-perfect) fields of positive characteristic
\cite[Th.~3.1 of Chap.~6, Page 65]{Oes--84}.

For \ref{2DEmotivchar0}, we check the contraposition.
Assume that there exist  $\G_1, \G_2$ and $\mu$
as in Definition \ref{almostdirectproduct}.
For $j \in \{1, 2\}$, denote by $\mathfrak g_j$ the Lie algebra of $\G_j$;
since $\G_j$ is normal in $\G$, it is a Lie ideal in $\mathfrak g$.
The differential 
$\mathfrak g_1 \times \mathfrak g_2 \longrightarrow \mathfrak g$
of $\mu$ at the identity
is a presentation of $\mathfrak g$ by a product.
\end{proof}

The following theorem is a variation on \cite[Proposition 3]{CoH--07}.

\begin{thm}
\label{inGalgnotproduct}
Let $k$ be a field of characteristic zero,
and $\G$ an algebraic $k$-group of positive dimension. 
Assume that $\G$ is connected,
and is not presentable by a product.
Then every Zariski dense subgroup of $\G(k)$ is not presentable by a product.
\end{thm}

In the situation of Theorem \ref{inGalgnotproduct},
observe that $\G(k)$ is Zariski-dense in itself, and therefore
is not presentable by a product as an abstract group.

Since Theorem \ref{inGalgnotproduct} applies to algebraic $k$-groups
of which the Lie algebras are not presentable by a product,
we indicate first some Lie algebras of this kind.
If $\mathfrak a$ is an ideal in a Lie algebra $\mathfrak g$,
we denote by $\mathfrak z_{\mathfrak g}(\mathfrak a)$
its centralizer $\{Y \in \mathfrak g \mid [Y,X] = 0 \hskip.2cm \forall X \in \mathfrak a \}$.

\begin{exe}
\label{exirrLiealg}
Let $k$ be a field of characteristic $0$.
Examples of Lie algebras over $k$ which \emph{are not} 
presentable by a product include:
\begin{enumerate}[noitemsep,label=(\arabic*)]
\item\label{1DEexirrLiealg}
The non-abelian soluble Lie algebra $\mathfrak{af}$ of dimension $2$, 
with basis $\{e,g\}$ such that $[g,e] = e$, 
also called the Lie algebra of the ``$ax+b$ group'',
or affine group of the line.\\
Indeed, the only non-zero ideals of $\mathfrak{af}$ are $ke$ and $\mathfrak{af}$ itself,
and their centralizers in $\mathfrak{af}$ are
$\mathfrak z_{\mathfrak{af}}(ke) = ke$ and $\mathfrak z_{\mathfrak{af}}(\mathfrak{af}) = \{0\}$.
It follows that there cannot exist two subalgebras of $\mathfrak{af}$
as in Definition \ref{almostdirectproduct}.
\item\label{2DEexirrLiealg}
The $3$-dimensional Lie algebra $\mathfrak{sol}$,
with basis $\{e,f,g\}$ such that $[g,e] = e$, $[g,f] = -f$, $[e,f] = 0$,
i.e.\ the Lie algebra of the group $\mathbf{Sol}(k)$,
consisting of matrices of the form
$\begin{pmatrix}
a & 0 & b
\\
0 & a^{-1} & c
\\ 
0 & 0 & 1
\end{pmatrix}$
with $a \in k^\times$ and $b,c \in k$.\\
Indeed, it is easy to check that the complete list of non-zero ideals of $\mathfrak{sol}$
contains exactly four items:
$\mathfrak{sol}$ itself, 
the two-dimensional derived algebra $[\mathfrak{sol},\mathfrak{sol}] = ke \oplus kf$
(which is maximal abelian),
and the two one-dimensional ideals $ke$ and $kf$.
Their centralizers in $\mathfrak{sol}$ are $\{0\}$ for the first
and $[\mathfrak{sol},\mathfrak{sol}]$ for the three others.
It follows again that there cannot exist two subalgebras of $\mathfrak{sol}$
as in Definition \ref{almostdirectproduct}.
\item\label{3DEexirrLiealg}
Simple Lie algebras.
\item\label{4DEexirrLiealg}
Semi-direct products $V^r \rtimes \mathfrak{g}$,
where $V$ is a finite dimensional $k$-vector space,
$\mathfrak g$ a simple Lie subalgebra of $\mathfrak{gl}(V)$
acting irreducibly on $V$, and $r$ a positive integer;
the space  $V^r = V \oplus \cdots \oplus V$ (with $r$ factors) 
is considered as an abelian Lie algebra
on which $\mathfrak g$ acts by the restriction of the natural diagonal action
of  $\mathfrak{gl}(V)$ on $V^r$.
\\
Indeed, every non-zero ideal of $V^r \rtimes \mathfrak g$
is either the full Lie algebra, with centre $\{0\}$, 
or a $ \mathfrak g$-invariant subspace of $V^r$, with centralizer $V^r$.
The conclusion follows as in \ref{1DEexirrLiealg} and \ref{2DEexirrLiealg} above.
A particular case of this example appears in Proposition \ref{ofBnotpbp}.
\end{enumerate}
\end{exe}

By Theorems \ref{motivchar0} and \ref{inGalgnotproduct},
the previous examples for Lie algebras provide the following examples for groups.

\begin{exe}
\label{exnppfirst}
The following groups are not presentable by products.
\begin{enumerate}[noitemsep,label=(\arabic*)]
\item\label{1DEexnppfirst}
Let 
$G = 
\begin{pmatrix}
\mathbf R^\times & \mathbf R 
\\
0 & 1 
\end{pmatrix}$
be the real affine group,
or more precisely the group of real points of the appropriate algebraic group.
Observe that $G$ is centreless, of dimension $2$.
As an algebraic group, it is connected, even if, as a real Lie group, 
its connected component
$ 
\begin{pmatrix}
\mathbf R_+^\times & \mathbf R 
\\
0 & 1 
\end{pmatrix}$
is of index two.
The group $G$ is not presentable by a product, 
by Example \ref{exirrLiealg}\ref{1DEexirrLiealg}.
\par

For every $n \in \Z$ with $\vert n \vert \ge 2$, the soluble Baumslag-Solitar group
$\BS(1,n) = \langle a,t \mid tat^{-1} = a^n \rangle$
is isomorphic to a Zariski dense subgroup of $G$:
\begin{equation*}
\BS(1,n) \, \approx \, 
\begin{pmatrix}
\mathbf n^{\mathbf Z} & \mathbf Z [1/n] 
\\
0 & 1 
\end{pmatrix} \, \subset \,
\begin{pmatrix}
\mathbf R^\times & \mathbf R 
\\
0 & 1 
\end{pmatrix} \, = \,
G .
\end{equation*}
It follows 
that $\BS(1,n)$ is not presentable by a product.
For all Baumslag-Solitar groups (soluble or not), 
see Theorem~\ref{BS} below.
\item\label{2DEexnppfirst}
Every Zariski dense subgroup 
in the $3$-dimensional solvable Lie group $\Sol = \mathbf{Sol}(\R)$
is not presentable by a product,
by Example \ref{exirrLiealg}\ref{2DEexirrLiealg}.
This recovers the case of lattices proved in Corollary~3.7 of~\cite{KoL--13}. 
Lattices are Zariski dense by  a version of the Borel density theorem 
proved for example as Corollary 1.2 in \cite{Sha--99}.
\item\label{3DEexnppfirst}
For every $n \ge 2$, the lattices
$\SL_n(\Z)$ and $\Sp_n(\Z)$ 
are Zariski dense in $\SL_n(\R)$ and $\Sp_n(\R)$ respectively,
and consequently are not presentable by products
by Example \ref{exirrLiealg}\ref{3DEexirrLiealg}.
\par
This carries over to every subgroup of $\SL_n(\R)$ containing $\SL_n(\Z)$,
such as $\SL_n(\Z [1/d])$, for an integer $d \ge 2$,
or $\SL_n(\Z[\sqrt d ])$, for a squarefree integer $d \ge 2$;
and similarly for subgroups of $\Sp_n(\R)$ containing $\Sp_n(\Z)$.
\par
The same holds for non-elementary Fuchsian groups,
which are Zariski dense in  $\operatorname{PSL}_2(\mathbf R)$,
and non-elementary Kleinian groups,
which are Zariski dense in  $\operatorname{PSL}_2(\mathbf C)$.
\item\label{4DEexnppfirst}
For every $n \ge 2$ and $r \ge 1$, the group
$(\Z^n)^r \rtimes \SL_n(\Z)$ is Zariski-dense in $(\R^n)^r \rtimes \SL_n(\R)$,
and consequently is not presentable by a product
by Example \ref{exirrLiealg}\ref{4DEexirrLiealg}.
Here $\SL_n(\R)$ acts naturally on $\R^n$, and diagonally on $(\R^n)^r$.
\end{enumerate}
\end{exe}

Conclusions similar to those of Theorem \ref{inGalgnotproduct} carry over to some situations
where $\G$ is an almost direct product, see Theorem \ref{irrZdensenot} below.
First we recall some background; 
see e.g.\ in \cite{BoT--65}, 
Items 0.6-0.7 for \ref{1DEbackground} to \ref{3DEbackground},
and 2.15 for \ref{5DEbackground} (and therefore \ref{4DEbackground}).

\begin{back}
\label{background}
Let $k$ be a field of charactersitic $0$
and $\G$ a linear algebraic $k$-group.
Assume that $\G$ is connected of positive dimension.
\begin{enumerate}[noitemsep,label=(\arabic*)]
\item\label{1DEbackground}
Let $\G_1, \hdots, \G_r$ be closed normal $k$-subgroups of $\G$.
Then $\G$ is the \textbf{almost direct product} of $\G_1, \hdots, \G_r$
if the product map $\G_1 \times \cdots \times \G_r \longrightarrow \G$ 
is surjective with finite kernel.
The group $\G$ is a \textbf{non-trivial almost direct product}
if this happens for some $r$-uple of closed normal $k$-subgroups of positive dimensions, with $r \ge 2$.
\item\label{2DEbackground}
The group $\G$ is \textbf{semisimple} if it does not contain any
non-trivial connected solvable normal subgroup.
\item\label{3DEbackground}
The group $\G$ is \textbf{almost $k$-simple} if
every closed normal subgroup of $\G$ distinct from $\G$ is finite.
\end{enumerate}
Assume moreover from now on that $\G$ is semisimple.
\begin{enumerate}[noitemsep,label=(\arabic*)]
\addtocounter{enumi}{3}
\item\label{4DEbackground}
The group $\G$ is the almost direct product of its 
minimal closed connected normal $k$-subgroups of positive dimensions.
These are called the \textbf{almost simple factors} of $\G$.
\item\label{5DEbackground}
More generally, let $H$ be a closed connected normal subgroup of $\G(k)$.
Let $\G_1, \hdots, \G_s$ be the almost simple factors of $\G$
such that $\G_i(k) \subset H$ for $i = 1, \hdots, s$.
Then $H$ is the group of $k$-points of a connected $k$-subgroup $\mathbf H$ of $\G$
which is the almost direct product of $\G_1, \hdots, \G_s$.
\end{enumerate}
\end{back}

The notion of irreducibility is standard for lattices in semisimple connected groups,
see e.g.\ \cite[5.20 \& 5.21]{Rag--72} and \cite[Page 133]{Mar--91}.
For Zariski dense subgroups of algebraic groups, it can be extended as follows:

\begin{defn}
\label{defirr}
Let $k$ be a field of characteristic $0$
and $\G$  a semisimple connected algebraic $k$-group of positive dimension.
A Zariski dense subgroup $\Gamma$ of $\G (k)$
is \textbf{reducible} if there exist normal $k$-subgroups
$\G_1, \G_2$ of $\G$ of positive dimensions
such that
\begin{enumerate}[noitemsep,label=(\arabic*)]
\item\label{1DEdefirr}
$\G_1(k) \G_2(k) = \G(k)$,
\item\label{2DEdefirr}
$\G_1(k) \cap \G_2(k)$ is finite,
\item\label{3DEdefirr}
the subgroup $(\Gamma \cap \G_1(k))(\Gamma \cap \G_2(k))$ is of finite index in $\Gamma$;
\end{enumerate}
and $\Gamma$ is \textbf{irreducible} otherwise.
\end{defn}

The next theorem generalizes Proposition 8.1 of \cite{KoL--13}.

\begin{thm}
\label{irrZdensenot}
If $k$ is a field of characteristic zero,
and $\G$ a semisimple connected algebraic $k$-group of positive dimension,
then every  irreducible Zariski dense subgroup  of $\G(k)$ is not presentable by a product.
\end{thm}

\begin{lem}
\label{LemmaG=G1G2}
Let $k$ be a field of chracteristic zero, 
$\G$ a connected algebraic $k$-group of positive dimension,
and $\Gamma$ a Zariski dense subgroup of $\G (k)$.
Let $\Gamma_1, \Gamma_2$ be two groups and 
$f : \Gamma_1 \times \Gamma_2 \longrightarrow \Gamma$
a homomorphism with image of finite index.
Denote by $\G_1$ the Zariski closure of $f(\Gamma_1 \times \{1\})$,
by $\G_2$ that of $f(\{1\} \times \Gamma_2)$, 
and by $\G_1^0, \G_2^0$ their connected components;
note they are $k$-subgroups of $\G$.

Then $\G_1^0$, $\G_2^0$ are closed normal $k$-subgroups of $\G$,
each one in the centralizer of the other,
and $\G_1^0 \G_2^0 = \G$.
\end{lem}
\begin{proof}
On the one hand,  
$\G_1 \G_2$ is closed in $\G$,
as the image of the morphism
\begin{equation*}
\mu \, : \, \G_1 \times \G_2 \longrightarrow \G, \hskip.2cm (g_1, g_2) \longmapsto g_1g_2 ,
\end{equation*}
see \cite[1.4(a) Page 47]{Bor--91}.
On the other hand, 
\begin{equation*}
[\G_1, \G_2] \, = \,  \{1\} , 
\end{equation*}
because 
\begin{equation*}
\left[ \hskip.1cm \overline{ f(\Gamma_1 \times \{1\})} ,   
\hskip.1cm  \overline{f(\{1\} \times \Gamma_2)}  \hskip.1cm \right] \, = \, 
\overline{  [f(\Gamma_1 \times \{1\}) , f(\{1\} \times \Gamma_2)] } \, = \,  \{1\} ,
\end{equation*} 
see \cite[2.4 Page 59]{Bor--91};
in particular, $\G_2$ [respectively $\G_1$] is in the centralizer in $\G$
of $\G_1$ [respectively $\G_2$].

It follows that $\G_1 \G_2$ is a subgroup of $\G$.
Since $\G$ is connected
and $f(\Gamma_1 \times \Gamma_2)$ is of finite index
in the Zariski-dense subgroup $\Gamma$,
the image $f(\Gamma_1 \times \Gamma_2) 
= f(\Gamma_1 \times \{1\}) f(\{1\}  \times \Gamma_2)$
is Zariski-dense in $\G$.
Hence $\G_1 \G_2  =  \G$.
Moreover, $\G_1, \G_2$ are normal in $\G$.

In algebraic groups, connected components are characteristic subgroups
\cite[1.4(b) Page 47]{Bor--91}.
Hence $\G_1^0, \G_2^0$ are normal subgroups in $\G$.
Using again the connectedness of $\G$,
we conclude that $\G_1^0 \G_2^0 = \G$.
\end{proof}
 
For the proofs of Theorems \ref{inGalgnotproduct} and \ref{irrZdensenot},
we use the same notation, $\G_1$, $\G_2$, $\G_1^0$, $\G_2^0$ 
as in the proof of Lemma \ref{LemmaG=G1G2}.

\begin{proof}[\textbf{Proof of Theorem \ref{inGalgnotproduct}}]
We prove the contraposition.
Suppose there exist a Zariski dense subgroup $\Gamma$ of $\G(k)$
and a presentation by a product $f : \Gamma_1 \times \Gamma_2 \longrightarrow \Gamma$.
This is the situation of the previous lemma, with moreover 
$f(\Gamma_1 \times \{1\}), f(\{1\} \times \Gamma_2)$ infinite,
therefore with $\G_1^0, \G_2^0$ of positive dimensions.
Hence the multiplication $\G_1^0 \times \G_2^0 \longrightarrow \G$,
$(g_1, g_2) \longmapsto g_1g_2$
is a presentation by a product, so that $\G$
is presentable by a product in the sense of Definition \ref{almostdirectproduct}.
\end{proof}

\begin{proof}[\textbf{Proof of Theorem \ref{irrZdensenot}}]
If $\G$ is almost $k$-simple, Theorem \ref{irrZdensenot}
is covered by Theorem  \ref{inGalgnotproduct}.
We can therefore assume that there exist an integer $t \ge 2$
and almost simple factors $\mathbf{H}_1, \hdots, \mathbf{H}_t$ of $\G$
such that $\G$ is the almost direct product of these.

We again prove the contraposition.
Let $\Gamma$ be a Zariski dense subgroup of $\G$
and $f : \Gamma_1 \times \Gamma_2 \longrightarrow \Gamma$
a presentation by a  product.
Let $\G_1^0, G_2^0$ be as in the previous proof.
By \ref{background},
upon reordering the $\mathbf{H}_j$, 
we can assume that
$\G_1^0 = \mathbf{H}_1 \cdots \mathbf{H}_s$ and 
$\G_2^0 = \mathbf{H}_{s+1} \cdots \mathbf{H}_t$
for some integer $s \in \{1, \hdots, t-1\}$.
The group $f(\Gamma_1 \times \{1\}) \cap \G_1^0(k)$ is of finite index in 
$f(\Gamma_1 \times \{1\})$,
and similarly for $f(\{1\} \times \Gamma_2) \cap \G_2^0(k)$ in $f(\{1\} \times \Gamma_2)$.
Hence $(\Gamma \cap \G_1^0(k))(\Gamma \cap \G_2^0(k))$ 
is of finite index in $\Gamma$,
and therefore $\Gamma$ is reducible.
\end{proof}

Example \ref{ExJustifiantIrred} below
is covered by Theorem \ref{irrZdensenot}, 
but not by Theorem \ref{inGalgnotproduct}.
We first record some remarks concerning the notion of irreducibility
for Zariski dense subgroups, as defined in \ref{defirr}.

\begin{rem}
\label{remreducible}
(1)
Let $k$ be a field of characteristic $0$ and $\G$
a semisimple algebraic $k$-group of positive dimension.
Let $\Gamma, \Delta$ be two Zariski dense subgroups  of $\G(k)$
such that $\Delta$ is a subgroup of finite index in $\Gamma$.
Then $\Delta$ is irreducible if and only if $\Gamma$ is irreducible.

(2)
Assume that $k = \mathbf R$.
Let $G$ be a  Lie group
of which the connected component is semisimple  without compact almost simple factor.
Assume moreover that $G$ is linear, and that $G$ is the group of real points
of a connected algebraic $\R$-group $\G$.
Let $\Gamma$ be a lattice in $G = \G(\R)$.
Then $\Gamma$ is Zariski dense in $G$, 
by the Borel density theorem \cite{Bor--60}.
The definitions are such that
irreducibility of $\Gamma$ as a lattice in $G$ in the usual sense \cite{Rag--72}
coincides with irreducibility of $\Gamma$ in $G$ in the sense of Definition \ref{defirr}.

(3)
In a connected algebraic group over a local field, 
a Zariski dense subgroup need not be a lattice.
Consider for example the field $\R$, the algebraic $\R$-group $\mathbf{SL}_2$,
and a lattice $\Gamma_g$  in $\SL_2(\R)$
which is the fundamental group of an oriented closed surface of genus $g \ge 2$;
the group of commutators of $\Gamma_g$ 
is Zariski dense in $\SL_2(\R)$, but is not a lattice.
For another class of examples, 
the real affine group does not have any lattice (because it is a non-unimodular group),
but all its non-abelian subgroups are Zariski dense; 
see Example \ref{exnppfirst}\ref{1DEexnppfirst}.
\end{rem}

We are grateful to Yves de Cornulier, 
who made us aware of the following lemma. 

\begin{lem}
Let $k$ be a field of characteristic zero,
$S$ the group of $k$-points of 
a simple connected algebraic $k$-group, with centre $C(S)$,
and $\Gamma$ a Zariski-dense subgroup of $S$.

For every infinite normal subgroup $N$  of $\Gamma$,
the centralizer $C_\Gamma(N)$ of $N$ in $\Gamma$ is contained in $C(S)$. 
In particular, $C_\Gamma(N)$ is finite.
\end{lem}

\begin{proof}
Denote by $\overline{N}^0$ the connected component of the Zariski closure of $N$ in $S$.
Since $N$ is infinite, $\overline{N}^0 \ne \{1\}$. 
Since $N$ is normal in $\Gamma$,
the group $\overline{N}^0$ is normal in $S$, hence it coincides with $S$.

Let $\gamma \in \Gamma$. If $\gamma$ commutes with every element of $N$,
it commutes with every element of $\overline{N}$.
Hence $C_\Gamma (N) \subset C_\Gamma \big(\overline N ^0(k)\big) 
= C_\Gamma (S) \subset C(S)$.
\end{proof}

\begin{exe}
\label{ExJustifiantIrred} 
Let $\Gamma$ be an irreducible Zariski-dense subgroup
of a semisimple connected algebraic $k$-group of the form $\G_1 \times \G_2$,
with $\G_1$ and $\G_2$ simple.
The following example shows (with $k = \R$) that $\Gamma$ need not be commensurable
to a Zariski-dense subgroup of $\G_1$ or of $\G_2$.

Let $F$ be a Zariski-dense subgroup of $\SL_2(\R)$ which is non-abelian and free.
Choose an epimorphism $\pi : F \twoheadrightarrow \Z$;
observe that $\ker (\pi)$ is also Zariski-dense in $\SL_2(\R)$.
Consider the direct product $G = \SL_2 (\R) \times \SL_2 (\R)$ and the fibre product
$$
\Gamma \, = \, \{ (x,y) \in \SL_2(\R) \times \SL_2(\R) \mid \pi(x) = \pi(y) \} .
$$
Since $\Gamma$ contains $\ker (\pi) \times \ker (\pi)$, 
it is a Zariski-dense subgroup of $G$.
Since $\Gamma \cap (\SL_2(\R) \times \{1\}) = \ker (\pi) \times \{1\}$
and  $\Gamma \cap (\{1\} \times \SL_2(\R)) = \{1\} \times \ker (\pi)$,
and the index of $\ker (\pi) \times \ker (\pi)$ in $\Gamma$ is infinite,
$\Gamma$ is an irreducible Zarski dense subgroup of $G$.
Hence Theorem \ref{irrZdensenot} applies to $\Gamma$. 

However $\Gamma$ is not commensurable to any of its two projections 
on the factors $\SL_2(\R)$, by the previous lemma.
Indeed, $\Gamma$ has an infinite normal subgroup $N := \ker (\pi) \times \{1\}$,
and $C_\Gamma(N)$ contains $\{1\} \times \ker (\pi)$,
in particular $C_\Gamma (N)$ is infinite.

Note that $\Gamma$ is a subdirect product of free groups,
i.e.\ is a member of a famous class of examples
studied among others by Bieri and Stallings.
See for example
Section 4.1 in \cite{BrM--09}.
\end{exe}

\section{A few specific applications}
\label{s:applic}

In this section we apply the criteria for presentability by products developed in the previous sections to a few classes of examples.

\subsection{Baumslag-Solitar groups}
\label{subsectionBS}

Recall that, for two integers $m, \ n \in \Z \smallsetminus \{0\}$,
the corresponding Baumslag-Solitar group is defined by the presentation
\begin{equation*}
\BS (m,n) \, = \, \langle s,t \mid ts^m t^{-1} = s^n \rangle .
\end{equation*}
This presentation has deficiency $1$, and it follows that the group itself
has deficiency $1$, cf.~\cite[Lemma 1.7]{Eps--61}.

\begin{thm}
\label{BS}
For $m,n \in \Z \smallsetminus \{0\}$, 
the group $\BS(m,n)$ is presentable by products
if and only if $\vert m \vert = \vert n \vert$.
\end{thm}

\begin{proof}
Since $\BS(n,m)$ and $\BS(m,n)$ are isomorphic,
we may assume without loss of generality that $\vert m \vert \le \vert n \vert$.

It is straightforward to check that $\BS (1,1)$ is isomorphic to $\Z^2$, 
and $\BS(1,-1)$ is isomorphic to the fundamental group of a Klein bottle;
in particular, these groups have an infinite centre,
and are therefore presentable by products. More cases in which $\BS(m,n)$
is presentable by products are given by the following well-known proposition, 
for which we have not been able to find a convenient reference;
compare with \cite[Lemma A.7]{Sou--01}.
(Proof continues after Lemma \ref{normalcyclicsubgroup}.)

\begin{prop}
\label{folkloreBS(m,+-m)}
For $m,n \in \Z \smallsetminus \{0\}$ with $\vert m \vert = \vert n \vert \ge 2$,
the Baumslag-Solitar group $\BS (m,n)$ has a subgroup of index $2\vert m \vert$
isomorphic to the direct product $\Z \times F_{2\vert m \vert - 1}$.
\end{prop}
\begin{proof}
Since $\BS(-m,-n)$ and $\BS(m,n)$ are isomorphic,
we assume that $m \ge 2$, without loss of generality.
We write $n = \eta m$, with $\eta \in \{1, -1\}$, and $B = \BS(m,\eta m)$.

Let $T$ be the subset $\{ t, sts^{-1}, s^2 t s^{-2}, \hdots, s^{m-1} t s^{-m+1} \}$ of $B$.
By the normal form theorem for HNN-extensions, see e.g.\ \cite[Section IV.2]{LyS--77}, 
the subgroup $F_{I}$ of $B$ generated by $T$ is free on $T$,
in particular is isomorphic to $F_m$.
Denote by $F_{II}$ the subgroup of $F_I$ consisting of
products of an even number of the generators in $T$ and their inverses;
the group $F_{II}$ is free of rank $2m - 1$.

Let $\Z_S$ denote the subgroup of $B$ generated by $s^m$.
It is an infinite cyclic group, and a normal subgroup of $B$.
Moreover, elements of $\Z_S$ commute with elements of $F_{II}$.
Hence the subgroup of $B$ generated by $s^m$ and $F_{II}$
is the direct product $\Z_S \times F_{II}$.
(If $\eta = 1$ then $\Z_S$ is central in $B$
and the subgroup of $B$ generated by $s^m$ and $F_I$ 
is the direct product $\Z_S \times F_{I}$;
but this does not hold when $\eta = -1$).

Denote by $C_m = \langle c \mid c^m = 1 \rangle$ the cyclic group of order $m$,
and by $D = \langle d\mid d^2 = 1 \rangle$ the group of order $2$.
Define an epimorphism
\begin{equation*}
\pi \, : \, B \longrightarrow C_m \times D
\hskip.2cm \text{by} \hskip.2cm
\pi(s) = (c,1) \hskip.2cm \text{and} \hskip.2cm \pi(t) = (1,d).
\end{equation*}
It is straightforward that $\Z_S \times F_{II} \subset \ker (\pi)$.
To end the proof, it suffices to check that $\Z_S \times F_{II} = \ker(\pi)$.

Let $g \in B$.
Using again the normal form theorem for HNN-extensions, we can write
\begin{equation*}
g \, = \,  s^{k_0} t^{\epsilon_1} s^{k_1} t^{\epsilon_2} s^{k_2} \cdots
t^{\epsilon_\ell} s^{k_\ell}
\end{equation*}
for some $k_0 \in \Z$, $\ell \ge 0$, $\epsilon_1, \hdots, \epsilon_\ell \in \{1, -1\}$,
$k_1, \hdots, k_\ell \in \{0, 1, \hdots, m-1\}$, and thus also
\begin{equation*}
\aligned
g \, = \,  &s^{k_0 + k_1 + \cdots + k_\ell} 
\left( s^{-k_1 - \cdots - k_\ell} t^{\epsilon_1} s^{k_1 + \cdots + k_\ell} \right)
\left( s^{-k_2 - \cdots - k_\ell} t^{\epsilon_2} s^{k_2 + \cdots + k_\ell} \right)
\\
&\cdots
\left( s^{-k_{\ell-1} - k_\ell} t^{\epsilon_{\ell-1}} s^{k_{\ell-1} + k_\ell} \right)
\left( s^{-k_\ell} t^{\epsilon_\ell} s^{k_\ell} \right) .
\endaligned
\end{equation*}
Since $t s^{m} = s^{\eta m} t$, 
we have 
\begin{equation*}
g \, = \, s^{j_0} \left( s^{j_1} t s^{-j_1} \right)^{\epsilon_1}
\left( s^{j_2} t s^{-j_2} \right)^{\epsilon_2}
\cdots
\left( s^{j_\ell} t s^{-j_\ell} \right)^{\epsilon_\ell}
\end{equation*}
for appropriate $j_0 \in \Z$ and $j_1, \hdots, j_\ell \in \{0, 1, \hdots, m-1 \}$.

Suppose now that $g \in \ker (\pi)$. 
Then $j_0 \in m \Z$ and $\ell$ is even,
hence $g \in \Z_S \times F_{II}$.
\end{proof}

We also need the following.
\begin{lem}
\label{normalcyclicsubgroup}
Let $\Gamma$ be a group that has a subgroup of finite index 
isomorphic to $F_k \times \Z$ for some $k \ge 2$.
Then $\Gamma$ has a normal infinite cyclic subgroup.
\end{lem}

\begin{proof}
Let $\Delta$ be a subgroup of finite index in $\Gamma$
isomorphic to the product of a free subgroup, on free generators $a_1, \hdots, a_k$,
and of an infinite cyclic group, on a generator $b$,
say $\Delta = F_{\langle a \rangle} \times \Z_b$. Observe that $\Z_b$ is the centre of $\Delta$.
Set $N = \bigcap_{\gamma \in \Gamma} \gamma \Delta \gamma^{-1}$;
it is a normal subgroup of finite index in $\Gamma$, contained in $\Delta$.
Let $k$ be the smallest positive integer such that $b^k \in N$;
then $b^k$ it a generator of the cyclic group $C :=N \cap \Z_b$.

On the one hand, $b^k$ is in the centre of $N$, indeed in that of $\Delta$.
On the other hand, any element in the centre of $N$ commutes
with appropriate powers of each of $a_1, \hdots, a_m$,
and is therefore in $C$.
Hence $C$ is the centre of $N$, and it follows that $C$ is normal in $\Gamma$.
\end{proof}

We can now complete the proof of Theorem \ref{BS}.
When $\vert m \vert < \vert n \vert$, we have two different ways to see that 
$\BS(m,n)$ is not presentable by products. In a first argument we distinguish
the cases $1 = \vert m \vert  < \vert n \vert$ and $1 < \vert m \vert  < \vert n \vert$. 
In the first case $\BS(m,n)$ is soluble, 
and is a Zariski dense subgroup of the real affine group considered 
in Example~\ref{exnppfirst}\ref{1DEexnppfirst}. 
Thus, it is not presentable by products by Theorem~\ref{inGalgnotproduct}.
In the second case $\BS(m,n)$ is a Powers group by Theorem~3 of~\cite{HaP--11},
and is therefore not presentable by products by Proposition~5.1 of~\cite{KoL--13}.

In a second argument we use Theorem~\ref{t:def}.
If $\BS(m,n)$ is presentable by products, 
then it has a finite index subgroup isomorphic to $F_k \times \Z$ for some $k \ge 1$. 
If $k = 1$, i.e. if $\BS(m,n)$ is virtually abelian,
then $\vert m \vert = \vert n \vert = 1$.
If $k \ge 2$, then $\BS(m,n)$ has a normal infinite cyclic subgroup 
by Lemma \ref{normalcyclicsubgroup},
and it follows that $\vert m \vert = \vert n \vert$ 
by a result of Moldavanskii \cite[Number 6.A]{Mol--91}.

Thus, as soon as $\vert m \vert < \vert n \vert$, we conclude that $\BS(m,n)$ is not presentable by products.
This completes the proof of Theorem \ref{BS}.
\end{proof}

\subsection{Three-manifold groups}
\label{subsection3fold}

We now discuss finitely presented fundamental groups of three-manifolds.

\begin{thm}
\label{t:3fold}
Let $\Gamma$ be the fundamental group of some connected three-manifold.
Assume that $\Gamma$ is infinite and finitely presented.
Then $\Gamma$ is presentable by a product if and only if it has a finite
index subgroup with infinite centre, 
equivalently if it is the fundamental group of a compact Seifert fibre space,
possibly with boundary.
\end{thm}

This extends the corresponding statement for closed three-manifold groups 
proved in \cite[Theorem~8]{KoN--13}. If the boundary 
of the Seifert fibre space is non-empty, it is a union of two-tori, cf.\ \cite[p.~430]{Sco--83}.

\begin{proof}
We refer the reader to Scott's survey \cite{Sco--83} 
for facts about Seifert fibre spaces used in this proof.

Suppose $\Gamma$ is an infinite finitely presented three-manifold group 
that is presentable by a product. 
By a result of Jaco~\cite{Jac--71}, 
finite presentability ensures that $\Gamma=\pi_1(M)$ for some compact three-manifold $M$, 
possibly with boundary. 
Passing to an index two subgroup we may assume that $M$ is orientable. 
Since manifolds double covered by Seifert fibre spaces are themselves Seifert fibre spaces, 
this does not affect the conclusion.

As $M$ is orientable, so is its boundary (if it is not empty).
Next, capping off an $S^2$ in the boundary of $M$ by a three-ball 
does not change the fundamental group, 
so we may assume that $M$ does not have any spherical boundary components.

If $M$ is closed, then the result was proved in \cite[Theorem~8]{KoN--13}. 
We may therefore assume that $M$ is compact 
with non-empty and non-spherical boundary. 
Now Epstein \cite[Lemma~2.2]{Eps--61} proved that
\begin{equation*}\label{eq:Ep}
\DEF (\pi_1(M))\geq 1-\chi (M) = 1-\frac{1}{2}\chi (\partial M) \ .
\end{equation*}
Since the Euler characteristic of the boundary is non-positive, 
the deficiency of $\pi_1(M)$ is positive, so that we can apply Theorem~\ref{t:def}. 
We conclude that the deficiency is $=1$, and so the boundary of $M$ consists of tori,
and that $\Gamma$ has a finite index subgroup with infinite centre. 
This means that $M$ is Seifert fibred; see \cite{Sco--83,Wal--67}.

Conversely, every group having a finite index subgroup with infinite centre 
is presentable by products.
\end{proof}

\begin{rem}
The case of a closed $M$ treated in \cite[Theorem~8]{KoN--13} 
cannot be dealt with by appealing to Theorem~\ref{t:def}, 
since for a closed aspherical three-manifold the fundamental group 
has deficiency $0$, by another result of Epstein \cite[Lemma~3.1]{Eps--61}.
\end{rem}

\begin{rem}
In the above proof, to show that for non-Seifert manifolds the fundamental 
group is not presentable by products, we distinguished between closed manifolds 
and those with non-empty boundary. Instead, we could make a different 
distinction into cases and use different arguments as follows. First of all, 
for $\Sol$-manifolds the fundamental group is not presentable by products
by Corollary~3.7 of~\cite{KoL--13}, reproved here in Example~\ref{exnppfirst}\ref{2DEexnppfirst}.
Second of all, for a compact three-manifold with infinite fundamental group which
is neither a $\Sol$ nor a Seifert manifold the fundamental group has the Powers 
property by Theorem~3 of~\cite{HaP--11}, and is therefore not presentable by products 
by Proposition~5.1 of~\cite{KoL--13}.
\end{rem}

\subsection{Coxeter groups}
\label{subsectionCoxetergroups}

The result of this section singles out another consequence of Theorem ~\ref{inGalgnotproduct}. 
It relies on \cite{BeH--04},
and is a variation on results contained in \cite{CoH--07}, as well as in \cite{Par--07}.

Let $(W,S)$ be a Coxeter system; here, we always assume that the generating set $S$ is finite.
We follow the terminology of \cite{Bou--68}.
Denote by 
\begin{enumerate}[noitemsep]
\item[-]
$\ell$ the size of $S$;
\item[-]
$E$ the real vector space $\mathbf R ^S$;
\item[-]
$B_S$ the Tits form, which is a symmetric bilinear form on $E$ associated to $(W,S)$;
\item[-]
$\ker (B_S)$ the kernel $\left\{ v \in V \mid 
B_S(v,w) = 0 \hskip.2cm \forall \hskip.1cm w \in V \right\}$ of $B_S$;
\item[-]
$\operatorname{Of}(B_S)$ the group of invertible linear transformations
$g \in \operatorname{GL}(E)$ such that 
$B_S(gv,gw) = B_S(v,w)$ for all $v,w \in E$ and $gv = v$ for all $v \in \ker (B_S)$;
\item[-]
$\mathfrak{of}(B_S)$ its Lie algebra;
\item[-]
$\sigma_S : W \longrightarrow \operatorname{Of}(B_S)$ the geometric representation of $W$,
which is known to be faithful with discrete image (a theorem due to Tits,
see $\S$~V.4 in \cite{Bou--68}).
\end{enumerate}
Let $p_B$ [respectively $q_B$] denote the maximal dimension of a subspace of $E$
to which the restriction of $B_S$ is positive definite 
[respectively negative definite], and set $r_B = \dim ( \ker(B_S) )$;
observe that $p_B+q_B+r_B = \ell$.
We have a semi-direct product of the form
\begin{equation*}
\mathfrak{of}(B_S) \, \simeq \, (\R^{p_B + q_B})^{r_B} \rtimes \mathfrak{so}(p_B, q_B) .
\end{equation*}
In other terms, $\mathfrak{of}(B_S)$ is isomorphic 
to the Lie algebra of matrices which can be written 
in bloc form relatively to the decomposition $\ell = p_B+q_B+r_B$ as
$
\begin{pmatrix}
r & s & 0 \\ t & u & 0 \\ x & y & 0
\end{pmatrix} 
$
with
\begin{equation*}
\begin{pmatrix}
r & s  \\ t & u 
\end{pmatrix} 
\in \mathfrak{so}(p_B, q_B)
\hskip.5cm \text{i.e.} \hskip.5cm
\begin{pmatrix}   r' & t'  \\ s' & u'    \end{pmatrix}
\begin{pmatrix}   I_p & \phantom{-}0  \\ 0\phantom{_p} & -I_q    \end{pmatrix}
+
\begin{pmatrix}   I_p & \phantom{-}0  \\  0\phantom{_p} & -I_q    \end{pmatrix}
\begin{pmatrix}  r & s  \\ t & u    \end{pmatrix}
\, = \, 0 
\end{equation*}
(where $r'$ is the transpose of the matrix of $r$),
and $\begin{pmatrix}   x & y  \end{pmatrix} \in (\R^{p_B + q_B})^{r_B}$.
\par

The following characterizations (at least the first two) are standard:
\begin{enumerate}[noitemsep,label=(\roman*)]
\item\label{iDE6C}
$W$ is finite if and only if $B_S$ is positive definite,  i.e.\  if $q_B = r_B = 0$;
\item\label{iiDE6C}
$W$ is affine, 
i.e.\ is infinite and contains a subgroup of finite index isomorphic to $\mathbf Z^d$
with $d \ge 1$, if and only if $B_S$ is positive and not definite,
i.e.\ if $q_B = 0$ and $r_B \ge 1$;
\item\label{iiiDE6C}
$W$ contains non-abelian free subgroups 
if and only if $B_S$ is not positive (definite or indefinite).
\end{enumerate}
The pair $(W,S)$ is \textbf{irreducible}
if the Coxeter graph of $(W,S)$ is connected. 
When this holds, we have moreover:
\begin{enumerate}[noitemsep,label=(\roman*)]
\addtocounter{enumi}{3}
\item\label{ivDE6C}
if $q_B = 0$, i.e.\ if $W$ is affine, then $r_B = 1$,  and $d = \ell-1$;
\item\label{vDE6C}
if $\ell = 2$, then $W$ is a dihedral group, either finite ($p_B = 2$) or infinite ($p_B = 1 = r_B$);
\item\label{viDE6C}
if $\ell = 3$, and $W$ is neither finite nor affine, then $(p_B, q_B, r_B) = (2, 1, 0)$;
\item\label{viiDE6C}
if $\ell \ge 4$, then $p_B \ge 3$.
\end{enumerate}
(See \cite[Chap.\ V, $\S$~4]{Bou--68}, 
and more precisely no 2 for \ref{ivDE6C}, 
no 9 for \ref{iiiDE6C}, 
and exercice 4 for \ref{viDE6C}.
For \ref{viiDE6C}, see Proposition 7 Page 31 of~\cite{Par--89} for $\ell = 4$, 
and observe that the same inequality for $\ell \ge 4$ follows. Since~\cite{Par--89}
is not widely available, we reproduce the relevant result in 
Proposition~\ref{propthseLP} in the appendix below.)

\begin{prop}
\label{ofBnotpbp}
Let $(W,S)$ be an irreducible Coxeter system, with $S$ finite and $W$ neither finite nor affine.
The real Lie algebra $\mathfrak{of}(B_S)$ is not presentable by a product.
\end{prop}

\begin{proof}
We continue with the notation above.
The hypothesis on $(W,S)$ imply that $p_B + q_B \ge 3$,
and that $(p_B, q_B)$ cannot be any of $(4,0), (2,2), (0,4)$.
As a consequence, the orthogonal Lie algebra $\mathfrak{so}(p_B, q_B)$ is simple;
moroever its action on $\R^{p_B+q_B}$ is irreducible.
Hence  $\mathfrak{of}(B_S)$ is a Lie algebra
as in Example \ref{exirrLiealg}\ref{4DEexirrLiealg},
and the proposition follows.
\end{proof}

Here is the main result of \cite{BeH--04}:

\begin{prop}
Let $(W,S)$ be a Coxeter system, with $S$ finite and $(W,S)$ irreducible.
Assume that $B_S$ is not positive.

Then the subgroup $\sigma_S (W)$ of $\operatorname{Of}(B_S)$ is Zariski dense.
\end{prop}


\begin{thm}
\label{Coxeter}
Let $(W,S)$ be an irreducible Coxeter system, with $S$ finite and $W$ neither finite nor affine.
Then $W$ is not presentable by a product.
\end{thm}

\begin{proof}
The two previous propositions and Theorem \ref{inGalgnotproduct}
imply that the intersection of $\sigma_S(W)$ with the connected component
of $\operatorname{Of}(B_S)$ is not presentable by a product.
Since this intersection if of finite index in $\sigma_S (W)$, 
the group $W$ itselfs not presentable by a product.
\end{proof}

Note that an affine  irreducible Coxeter group is presentable by a product,
because it has a subgroup of finite index isomorphic to $\Z^d$, 
for some $d \ge 1$.

\bigskip

\appendix

\section{Signatures of Coxeter systems with $4$ generators
\\
(from the 1989 thesis of L.~Paris)}

\vskip.2cm

We reproduce here the proof of 
Proposition 7, Page 31, of \cite{Par--89}.

Let $B = \left( B_{i,j} \right)_{1 \le i,j \le n}$ be a symmetric matrix with real coefficients, of size $n$,
where $B_{i,i} = 1$ and $-1 \le B_{i,j} \le 0$ for $i,j \in \{1, \hdots, n\}$ with $i \ne j$.
Define $p,q,r$ for $B$ as $p_B, q_B, r_B$ for $B_S$ in Subsection \ref{subsectionCoxetergroups} above.
Recall that $p+q+r=n$.

\begin{lem}
\label{PFn=3}
If $n = 3$, then $p \ge 2$. If moreover $\det (B) > 0$, then $p = 3$.
\end{lem}

\begin{proof}
(Compare with \cite[Ch.\ V, $\S$ 4, Exercice 4]{Bou--68}.) 
If $B_{1,2} = B_{1,3} = B_{2,3} = -1$, a computation shows that the eigenvalues of $B$
are $-1, 2, 2$, so that $p=2$, $q=1$, and $\det (B) < 0$. 
\par

Suppose now that at least one of the off-diagonal terms is not $-1$;
upon permuting $1, 2, 3$, we can assume 
that $-1 < B_{1,2} \le 0$. We have
$$
\begin{pmatrix} x & y & 0 \end{pmatrix}
B
\begin{pmatrix} x \\ y \\ 0 \end{pmatrix}
\, = \, (x + B_{1,2} y)^2 + (1 - (B_{1,2})^2)y^2 \, > \, 0
$$
for all $(x,y) \in \R^2 \smallsetminus \{(0,0) \}$; hence $p \ge 2$.
There are three possible cases:
either $p=3$ and $\det (B) > 0$,
or $p=2, r=1$ and $\det (B) = 0$,
or $p=2, q=1$ and $\det (B) < 0$.
\end{proof}

Recall that $B$ is \emph{reducible} (\`a la Perron-Frobenius) if there exists
a non-trivial partition $\{1, \hdots, n\} = I \sqcup J$ such that $B_{i,j} = 0$ for all $i \in I, j \in J$,
and \emph{irreducible} otherwise.

\begin{lem}
\label{PFn=4}
If $n=4$ and if $B$ is irreducible, then $p \ge 3$.
\end{lem}

\begin{proof}
Since $B$ is irreducible, upon permuting the indices $1,2,3,4$,
we may assume that $B_{1,2} \ne 0$ and $B_{1,3} \ne 0$.
\par
Set $B' =  \left( B_{i,j} \right)_{2 \le i,j \le 4}$, and let $p', q', r'$ be defined for $B'$
as $p,q,r$ for $B$. Observe that $p \ge p'$ and $q \ge q'$. 
By expanding along the first column, we have
$$
\aligned
\det (B) \, &= \, 
\begin{vmatrix} 
1 & B_{1,2} & B_{1,3} & B_{1,4} \\
B_{1,2} & 1 & B_{2,3} & B_{2,4} \\ 
B_{1,3} & B_{2,3} & 1 & B_{3,4} \\ 
B_{1,4} & B_{2,4} & B_{3,4} & 1 
\end{vmatrix}
\\
\, &= \, 
\begin{vmatrix} 1 & B_{2,3} & B_{2,4} \\ B_{2,3} & 1 & B_{3,4} \\ B_{2,4} & B_{3,4} & 1 \end{vmatrix}
\, - \, B_{1,2}
\begin{vmatrix} B_{1,2} & B_{1,3} & B_{1,4} \\ B_{2,3} & 1 & B_{3,4} \\ B_{2,4} & B_{3,4} & 1 \end{vmatrix}
\\
& \hskip.5cm \, + \, B_{1,3}
\begin{vmatrix} B_{1,2} & B_{1,3} & B_{1,4} \\ 1 & B_{2,3} & B_{2,4} \\ B_{2,4} & B_{3,4} & 1 \end{vmatrix}
\, - \, B_{1,4} 
\begin{vmatrix} B_{1,2} & B_{1,3} & B_{1,4} \\ 1 & B_{2,3} & B_{2,4} \\ B_{2,3} & 1 & B_{3,4} \end{vmatrix}
\\
\, &= \,
\det (B') 
\\
& \hskip.5cm \, - \, 
(B_{1,2})^2 \begin{vmatrix} 1 & B_{3,4} \\ B_{3,4} & 1 \end{vmatrix}
\, - \,  (B_{1,3})^2 \begin{vmatrix} 1 & B_{2,4} \\ B_{2,4} & 1 \end{vmatrix}
\\
& \hskip2cm \, - \, 
 (B_{1,4})^2 \begin{vmatrix} 1 & B_{2,3} \\ B_{2,3} & 1 \end{vmatrix}
\\
& \hskip.5cm \, + \, 
2 B_{1,2} B_{1,3} \begin{vmatrix} B_{2,3} & B_{2,4} \\ B_{3,4} & 1 \end{vmatrix}
\, + \,
2 B_{1,2} B_{1,4} \begin{vmatrix} B_{2,4} & B_{2,3} \\ B_{3,4} & 1 \end{vmatrix}
\\
& \hskip2cm \, + \, 
2 B_{1,3} B_{1,4} \begin{vmatrix} B_{3,4} & B_{2,3} \\ B_{2,4} & 1 \end{vmatrix} .
\endaligned
$$
In case $\det (B') > 0$, we have $p' = 3$ by Lemma \ref{PFn=3}, and thus $p \ge 3$.
\par
We assume from now on that $\det (B') \le 0$.
The minor expansion above shows that $\det (B)$ is the sum of seven non-positive terms.
\par
If we had $\det(B) = 0$, each of these seven terms should be zero.
In particular, we would have
$(B_{1,2})^2 \begin{vmatrix} 1 & B_{3,4} \\ B_{3,4} & 1 \end{vmatrix} = 0$,
hence $B_{3,4} = -1$, 
and $(B_{1,3})^2 \begin{vmatrix} 1 & B_{2,4} \\ B_{2,4} & 1 \end{vmatrix} = 0$,
hence $B_{2,4} = -1$.
We would then have
$B_{1,2} B_{1,3} \begin{vmatrix} B_{2,3} & B_{2,4} \\ B_{3,4} & 1 \end{vmatrix} < 0$,
and this is impossible.
\par
It follows that $\det (B) < 0$; hence $r=0$, and $q$ is odd.
Since we know that $p \ge p' \ge 2$ by Lemma \ref{PFn=3},
we have $q = 1$ and therefore $p=3$.
\end{proof}

This implies:

\begin{prop}[\cite{Par--89}]
\label{propthseLP}
Let $B = \left( B_{i,j} \right)_{1 \le i,j \le n}$ be a symmetric matrix with real coefficients, of size $n$,
with $B_{i,i} = 1$ and $-1 \le B_{i,j} \le 0$ for $i,j \in \{1, \hdots, n\}$ with $i \ne j$.
Let $p, q, r$ be as above; assume that $n = p+q+r \ge 4$.
\par

If $B$ is irreducible, then $p \ge 3$.
\end{prop}

\end{document}